\documentclass[12pt]{article}
\usepackage{amssymb,amsmath,amsfonts}
\usepackage{amsthm}
\usepackage{cite}
\usepackage[textwidth=160mm,textheight=220mm]{geometry}

\usepackage{graphicx}
\usepackage[dvips]{epsfig}
\usepackage{subfigure}
\usepackage[active]{srcltx}

\newtheorem{lemm}{Lemma}[section]
\newtheorem{theo}{Theorem}[section]

\newtheorem{prop}{Proposition}[section]

\theoremstyle{definition}
\newtheorem{Defi}{Definition}[section]

\newtheorem{obs}{Remark}[section]

\numberwithin{equation}{section}

\newcommand{\cabr}{C^2\big([a,b];\RR\big)}
\newcommand{\cloc}{C^1_{loc}\big((0,1/2];\RR\big)}

\newcommand{\lopr}{L\big([0,1/2];\RR\big)}
\newcommand{\lnek}{L^{+\infty}\big([0,1/2];\RR\big)}

\newcommand{\labrp}{L\big([a,b];\RR_+\big)}

\newcommand{\car}{\operatorname{Car}\big([a,b]\times\RR^2;\RR\big)}
\newcommand{\for}{\qquad\mbox{for~}\,}
\newcommand{\mfor}{\quad\mbox{for~}\,}
\newcommand{\forae}{\qquad\mbox{for~a.~e.~}\,}
\newcommand{\tin}{t\in [a,b]}
\newcommand{\acloc}{C^2_{loc}\big((0,1/2];\RR\big)}

\newcommand{\aclocab}{AC^1_{loc}\big([a,b]\setminus\{t_1,\dots,t_n\};\RR\big)}
\newcommand{\acabr}{AC^{1}\big([a,b];\RR\big)}

\def\RR{\mathbb{R}}

\def\NN{\mathbb{N}}

\begin{document}

\title{Existence and nonexistence results for a singular boundary value problem arising in the theory of epitaxial growth}

\bigskip

\author{Carlos Escudero\footnote{Supported by projects MTM2010-18128, RYC-2011-09025 and SEV-2011-0087.}
\\ [4mm] {\small Departamento de Matem\'aticas
\& ICMAT (CSIC-UAM-UC3M-UCM),}
\\{\small Universidad Aut\'onoma de Madrid, E-28049 Madrid, Spain}\\
[4mm] Robert Hakl\footnote{Supported by project RUO:67985840.}\\
[4mm] {\small Institute of Mathematics AS CR, \v Zi\v zkova 22,
616 62 Brno, Czech Republic}\\ [4mm]
Ireneo Peral\footnote{Supported by project MTM2010-18128.}\\
[4mm] {\small Departamento de Matem\'aticas,} \\{\small Universidad
Aut\'onoma de Madrid, E-28049 Madrid, Spain} \\ [4mm] Pedro J.
Torres\footnote{Supported by project MTM2011-23652.}\\
[4mm] {\small Departamento de Matem\'atica Aplicada,}
\\ {\small Universidad de Granada, E-18071 Granada, Spain.}
}

\date{}

\maketitle

\begin{abstract}
The existence of stationary radial solutions to a partial differential equation arising
in the theory of epitaxial growth is studied. Our
results depend on the size of a parameter that plays the role of
the velocity at which mass is introduced into the system. For small values
of this parameter we prove existence of solutions to this boundary value problem. For large
values of the same parameter we prove nonexistence of solutions. We also provide rigorous
bounds for the values of this parameter which separate
existence from nonexistence. The proofs come as a combination of several differential
inequalities and the method of upper and lower functions.

\end{abstract}

\section{Introduction}

Epitaxial growth is a technique used in the semiconductor industry
for the growth of thin films~\cite{barabasi}. It is employed for growing
crystal structures by means of the deposition of a given
material under high vacuum conditions. In epitaxial growth it is
quite usual finding a mounded structure generated along the
surface evolution rather than a flat surface~\cite{lengel}.
A number of models has been considered in order to explain this
phenomenology. These models have been usually introduced either as a
discrete probabilistic system or as a differential equation~\cite{barabasi}.
In this work we are interested in this second type of modeling approach.

Here we will focus on the rigorous mathematical analysis
of ordinary differential equations related to
models which have been introduced in the context of epitaxial
growth. The mathematical description of epitaxial growth uses the function
\begin{equation}
\phi: \Omega \subset \mathbb{R}^2 \times \mathbb{R}^+ \rightarrow
\mathbb{R},
\end{equation}
which describes the height of the growing interface in the spatial
point $x \in \Omega \subset \mathbb{R}^2$ at time $\mathfrak{t} \in
\mathbb{R}^+$. Although this theoretical framework can be extended
to any spatial dimension $N$, we will concentrate here on the
physical situation $N=2$. A basic modeling assumption is of
course that $\phi$ is a well defined function, a fact that holds in a
reasonably large number of cases~\cite{barabasi}. The macroscopic
description of the growing interface is given by a partial
differential equation for $\phi$ which is usually postulated using
phenomenological and symmetry arguments~\cite{barabasi,marsili}.
There are many such equations in the theory of non-equilibrium
surface growth. We will focus on the following one
\begin{equation}
\label{parabolic2} \partial_{\mathfrak{t}} \phi = \kappa_1 \, \det \left( D^2 \phi \right) -
\kappa_2 \, \Delta^2 \phi + \xi(x,\mathfrak{t}).
\end{equation}
This partial differential equation has been considered in the physical literature
as a possible continuum description of epitaxial growth~\cite{escudero,esckor}.
At the mathematical level we can consider it a parabolic problem whose evolution
is dictated by a Monge-Amp\`{e}re term stabilized
by a fourth order viscosity.

In this work we are concerned with the
stationary version of (\ref{parabolic2}), which reads
\begin{equation}\label{Pro0}
\left\{\begin{array}{rcl} \Delta^2 \phi &=& \text{det} \left( D^2 \phi
\right) +\lambda F, \qquad x\in \Omega\subset\mathbb{R}^2, \\
\text{boundary}&\,& \text{ conditions,}
\end{array}
\right.
\end{equation}
after getting rid of the equation constant parameters by means of
a trivial re-scaling of field and coordinates. Our last assumption
is that the forcing term $F=F(x)$ is time independent. The
constant $\lambda$ is a measure of the speed at
which new particles are deposited, and for physical reasons we
assume $\lambda \ge 0$ and $F(x) \ge 0$. We will devote our
efforts to rigorously clarify the existence and
nonexistence of solutions to this elliptic problem when set on a
radially symmetric domain.

\section{Radial Solutions}
\label{radialproblems}

As already outlined in the previous section,
our goal will be determining the existence of radially symmetric solutions of boundary
value problem~(\ref{Pro0}). We set the problem on the unit disk.
In a previous work we have numerically found, for a constant
forcing term $F(r) \equiv 1$, where $r$ is the radial coordinate, existence of two solutions for
small values of $\lambda$ and nonexistence of any solution for
sufficiently large values of $\lambda$~\cite{ehpt}. In this paper we will
still assume the forcing term is constant, and we will find
rigorous bounds for the values of $\lambda$ separating existence
from nonexistence. We note that we could also build the existence/nonexistence
theory for arbitrary forcing terms $F(r)$ employing the same methods. However, in this case, of course,
we would lose accuracy on the estimates of the bounds of $\lambda$.

In more precise terms, we look for
solutions of the form $\phi=\tilde{\phi}(r)$ where $r = |x|$.
By means of a direct substitution we find
\begin{equation}
\label{fullradial} \frac{1}{r} \left\{ r \left[ \frac{1}{r} \left(
r \tilde{\phi}' \right)' \right]' \right\}' = \frac{1}{r} \,
\tilde{\phi}' \tilde{\phi}'' + \lambda F(r),
\end{equation}
where $'=\frac{d}{dr}$. In the first case we consider homogeneous Dirichlet boundary
conditions for problem~(\ref{Pro0}). This translates to the conditions $\tilde{\phi}'(0)=0$,
$\tilde{\phi}(1)=0$, $\tilde{\phi}'(1)=0$, and $\lim_{r \to 0} r
\tilde{\phi}'''(r)=0$; the first one imposes the existence of an extremum at
the origin and the second and third ones are the actual boundary
conditions. The fourth boundary condition is technical and imposes
higher regularity at the origin. If this condition were removed
this would open the possibility of constructing functions $\tilde{\phi}(r)$
whose second derivative had a peak at the origin. This would in
turn imply the presence of a measure at the origin when
calculating the fourth derivative of such an $\tilde{\phi}(r)$, so this type
of function cannot be considered as an acceptable solution
of~(\ref{fullradial}) whenever $F(r)$ is a function. Throughout
this work we will assume $F \equiv 1$ as already mentioned.

Integrating once equation \eqref{fullradial} against measure $r \,
dr$ and using boundary condition $\lim_{r \to 0} r \tilde{\phi}'''(r)=0$
yields
\begin{equation}
r \left[ \frac{1}{r} \left( r \tilde{\phi}' \right)' \right]' =
\frac{1}{2} (\tilde{\phi}')^2 + \frac{1}{2} \lambda r^2.
\end{equation}
By changing variables $w=r \tilde{\phi}'$ we find the equation
\begin{equation}\label{numerics}
w'' -\frac{1}{r} \, w' = \frac{1}{2} \, \frac{w^2}{r^2} +
\frac{1}{2} \, \lambda \, r^2,
\end{equation}
subject to the conditions $w'(0) = 0$ and $w(1) = 0$.
This is the equation that has been numerically integrated in~\cite{ehpt};
now we proceed to summarize these results.
We observed that for $\lambda=0$ there
are one trivial and one non-trivial solutions. For $0 < \lambda <
\lambda_0$ there are two non-trivial solutions which approach each
other for increasing $\lambda$. For $\lambda > \lambda_0$ no more solutions were
numerically found. The critical value of $\lambda$ was numerically
estimated to be $\lambda_0 \approx 169$. These results suggest no
solutions exist for large enough $\lambda$.

In the second case we reconsider problem~(\ref{Pro0}) but
this time subject to homogeneous Navier boundary conditions.
For the radial case we
start as above, with equation~(\ref{numerics}), but now
with the condition $w(1)=w'(1)$ arbitrary instead of $w(1) = 0$, what corresponds to homogeneous Navier
boundary conditions. The results for this second case were qualitatively
similar to those of the first case but quantitatively rather different~\cite{ehpt}.
In particular, we numerically observed
that for $\lambda=0$ there are one trivial and one non-trivial
solutions. For $0 < \lambda < \lambda_0$ there are two non-trivial
solutions which approach each other for increasing $\lambda$. For
$\lambda > \lambda_0$ no more solutions were numerically found. The
critical value of $\lambda$ was numerically estimated to be
$\lambda_0 \approx 11.34$.

\section{Statement of the Boundary Value Problem}
\label{sec_1}

The rest of this work is devoted to rigorously justify the numerical facts
summarized in the previous section. We will show how to proof existence of solutions
for small $\lambda$ and nonexistence for large $\lambda$, as well as provide rigorous
bounds for the values of this parameter which separate existence from nonexistence.
Our first step will be recasting the differential equation under study into a
form more suitable for its mathematical analysis.

Let us consider again equation~(\ref{numerics}), which we write
now in the following form
\begin{equation} \label{1}
r^2w''-rw'=\frac{\lambda}{2}r^4+\frac{w^2}{2}\for r\in (0,1]
\end{equation}
with $\lambda\geq 0$, together with the boundary conditions
\begin{gather}
\label{2}
w'(0)=0,\\
\label{3} w(1)=0,
\end{gather}
corresponding to Dirichlet boundary conditions, respectively the
conditions \eqref{2} and
\begin{equation}
\label{4} w(1)=w'(1),
\end{equation}
corresponding to Navier boundary conditions. The transformation
\begin{equation}
\label{5} t=\frac{r^2}{2},\qquad u(t)=w(r)
\end{equation}
leads to the equation
\begin{equation}
\label{6} u''=\frac{u^2}{8t^2}+\frac{\lambda}{2}\for t\in(0,1/2]
\end{equation}
and the conditions
\begin{gather}
\label{7}
\lim_{t\to 0_+}\sqrt{t}\, u'(t)=0,\\
\label{8} u(1/2)=0,
\end{gather}
and
\begin{equation}
\label{9} u(1/2)=u'(1/2).
\end{equation}
Obviously, using the transformation \eqref{5} one can easily check
that problems \eqref{1}--\eqref{3} (resp. \eqref{1}, \eqref{2},
\eqref{4}) and \eqref{6}--\eqref{8} (resp. \eqref{6}, \eqref{7},
\eqref{9}) are equivalent. By a solution to \eqref{6} we
understand a function $u$ belonging to the space $\acloc$ of
functions $u:(0,1/2]\to\RR$ such that $u\in\cabr$ for every
compact interval $[a,b]\subseteq (0,1/2]$, and satisfying
\eqref{6}.

\section{Main Results}
\label{sec_2}

We start with the case of the Dirichlet conditions.

\begin{theo}
\label{t1} There exists a real number $\lambda_0>0$ such that
problem \eqref{1}--\eqref{3} is solvable for every $\lambda\in
[0,\lambda_0)$ and there is no solution to problem
\eqref{1}--\eqref{3} if $\lambda>\lambda_0$. Furthermore, every
solution $w$ to \eqref{1}--\eqref{3} satisfies
\begin{equation}
\label{87} w(r)\leq 0 \for r\in (0,1],\qquad \lim_{r\to
0_+}w(r)=0.
\end{equation}
\end{theo}

The information contained in this Theorem is complementary to the
existence and multiplicity results obtained in
section~\ref{radialproblems} by means of variational arguments. On
the other hand, the following Proposition gives us a localization
of the critical value of $\lambda$ for problem
\eqref{1}--\eqref{3}.

\begin{prop}
\label{p1} The number $\lambda_0$ from Theorem~\ref{t1} admits the
estimates
$$
144\leq \lambda_0\leq 307.
$$
\end{prop}

The situation is analogous for the Navier conditions.
\begin{theo}
\label{t2} There exists a real number $\lambda_0>0$ such that
problem \eqref{1}, \eqref{2}, \eqref{4} is solvable for every
$\lambda\in [0,\lambda_0)$ and there is no solution to problem
\eqref{1}, \eqref{2}, \eqref{4} if $\lambda>\lambda_0$.
Furthermore, every solution $w$ to \eqref{1}, \eqref{2}, \eqref{4}
satisfies \eqref{87}.
\end{theo}

The following Proposition gives us a localization of the critical
value of $\lambda$ for problem \eqref{1}, \eqref{2}, \eqref{4}.

\begin{prop}
\label{p2} The number $\lambda_0$ from Theorem~\ref{t2} admits the
estimates
$$
9\leq \lambda_0\leq \frac{128}{11}=11.\overline{63}.
$$
\end{prop}

By comparing these estimates with the numerical results, one
observes that in both cases the rigorous bounds capture the order of
magnitude of the critical value of the parameter. Estimates for the
Navier problem are more accurate, and in this case the upper bound
is rather precise.

The rest of this section is devoted to prove the latter results by
means of a combination of techniques for differential inequalities
and the method of upper and lower functions.

\section{Auxiliary Propositions}
\label{sec_3}

In what follows, we establish some basic properties of a function
$u\in\acloc$ satisfying the inequality
\begin{equation}
\label{10} u''(t)\geq \frac{u^2(t)}{8t^2}+\frac{\lambda}{2}\for
t\in (0,1/2].
\end{equation}

\begin{lemm}
\label{l1} Let $u\in\acloc$ satisfy \eqref{7} and \eqref{10}. Then
\begin{equation}
\label{11} \lim_{t\to 0_+}u(t)=0.
\end{equation}
\end{lemm}

\begin{proof}
From \eqref{10} it follows that $u'$ is a non-decreasing function.
Therefore, there exists $t_0\in (0,1/2]$ such that either
$$
u'(t)\geq 0 \for t\in (0,t_0]
$$
or
$$
u'(t)<0 \for t\in (0,t_0].
$$
In both cases, the function $u$ is monotone on the interval
$(0,t_0]$. Therefore, there exists a (finite or infinite) limit
$u(0+)$. Assume that \eqref{11} does not hold. Then there exist
$t_1\in (0,1/2]$ and $\delta>0$ such that
$$
u^2(t)\geq \delta\for t\in (0,t_1].
$$
Thus the integration of \eqref{10} from $t$ to $t_1$ yields
\begin{equation}
\label{12} u'(t_1)-u'(t)\geq
\frac{\delta}{8}\left(\frac{1}{t}-\frac{1}{t_1}\right)\for t\in
(0,t_1].
\end{equation}
Multiplying both sides of \eqref{12} by $\sqrt{t}$ and applying a
limit as $t$ tends to zero we obtain
$$
-\lim_{t\to 0_+}\sqrt{t}\, u'(t)\geq \frac{\delta}{8}\lim_{t\to
0_+}\frac{1}{\sqrt{t}}=+\infty,
$$
which contradicts \eqref{7}.
\end{proof}

\begin{obs}
\label{r1} Note that every function $u\in\acloc$ satisfying
\eqref{7} and \eqref{10} satisfies also
\begin{equation}
\label{13} \lim_{t\to 0_+}\frac{u(t)}{\sqrt{t}}=0.
\end{equation}
Indeed, the equality \eqref{13} follows from Lemma~\ref{l1} and
the l'Hospital rule.

On the other hand, if \eqref{13} holds then \eqref{11} is
fulfilled.
\end{obs}

\begin{lemm}
\label{l2} Let $u\in\acloc$ satisfy \eqref{8}, \eqref{11}, and
\begin{equation}
\label{14} u''(t)\geq 0\for t\in (0,1/2].
\end{equation}
Then
\begin{equation}
\label{15} u(t)\leq 0\for t\in (0,1/2].
\end{equation}
\end{lemm}

\begin{proof}
Assume on the contrary that there exists $t_0\in (0,1/2]$ such
that $u(t_0)>0$. According to \eqref{11} there exists $t_1\in
(0,t_0)$ such that
\begin{equation}
\label{16} u(t_1)<u(t_0).
\end{equation}
On the other hand, in view of \eqref{8} and \eqref{14} we have
$$
u(t)\leq \frac{u(t_1)(1/2-t)}{1/2-t_1}\for t\in[t_1,1/2].
$$
However, the latter inequality contradicts \eqref{16}.
\end{proof}

\begin{lemm}
\label{l3} Let $u\in\acloc$ satisfy \eqref{9}, \eqref{11}, and
\eqref{14}. Then \eqref{15} holds.
\end{lemm}

\begin{proof}
First we will show that
\begin{equation}
\label{17} u(1/2)\leq 0.
\end{equation}
Assume on the contrary that \eqref{17} does not hold. Then,
according to \eqref{11}, there exists $t_0\in(0,1/2)$ such that
\begin{equation}
\label{18} u(t_0)<u(1/2)(1/2+t_0).
\end{equation}
Obviously, \eqref{14} yields
$$
u'(1/2)\geq \frac{u(1/2)-u(t_0)}{1/2-t_0},
$$
whence, in view of \eqref{9}, we obtain
$$
u(1/2)\geq \frac{u(1/2)-u(t_0)}{1/2-t_0},
$$
However, the latter inequality contradicts \eqref{18}. Therefore,
the inequality \eqref{17} holds.

Finally, \eqref{9}, \eqref{14}, and \eqref{17} imply
$$
u'(t)\leq 0\for t\in (0,1/2],
$$
which together with \eqref{11} results in \eqref{15}.
\end{proof}

\begin{lemm}
\label{l3.5} Let $u\in\acloc$ satisfy \eqref{13}--\eqref{15}. Then
\eqref{7} is fulfilled.
\end{lemm}

\begin{proof}
In view of Remark~\ref{r1} the equality \eqref{11} holds.
Therefore we have
\begin{equation}
\label{20} u(t)=\int_0^t u'(s)ds\for t\in (0,1/2].
\end{equation}
According to \eqref{14}, from \eqref{20} it follows that
\begin{equation}
\label{21} u(t)\leq tu'(t)\for t\in (0,1/2].
\end{equation}
Moreover, \eqref{11}, \eqref{14}, and \eqref{15} imply the
existence of $t_0\in (0,1/2]$ such that $u'(t)\leq 0$ for $t\in
(0,t_0]$ which, together with \eqref{21}, results in
\begin{equation}
\label{22} \frac{u(t)}{\sqrt{t}}\leq \sqrt{t}\, u'(t)\leq 0\for
t\in (0,t_0].
\end{equation}
Now we get \eqref{7} from \eqref{22} using \eqref{13}.
\end{proof}

\begin{lemm}
\label{l4} The problems \eqref{6}--\eqref{8} (resp. \eqref{6},
\eqref{7}, \eqref{9}) and \eqref{6}, \eqref{8}, \eqref{13} (resp.
\eqref{6}, \eqref{9}, \eqref{13}) are equivalent.
\end{lemm}

\begin{proof}
Let $u$ be a solution to \eqref{6}--\eqref{8} (resp. \eqref{6},
\eqref{7}, \eqref{9}). Then, according to Remark~\ref{r1}, $u$ is
also a solution to \eqref{6}, \eqref{8}, \eqref{13} (resp.
\eqref{6}, \eqref{9}, \eqref{13}).

On the other hand, let $u$ be a solution to \eqref{6}, \eqref{8},
\eqref{13} (resp. \eqref{6}, \eqref{9}, \eqref{13}). Then, in view
of \eqref{6}, Remark~\ref{r1} and Lemma~\ref{l2} (resp.
Lemma~\ref{l3}), the inequalities \eqref{14} and \eqref{15} hold.
Thus we get that \eqref{7} is fulfilled from Lemma~\ref{l3.5}.
\end{proof}

\begin{obs}
\label{r2} It follows from Lemmas~\ref{1}--\ref{3} that every
solution $u$ to \eqref{6}--\eqref{8} (resp. \eqref{6}, \eqref{7},
\eqref{9}) satisfies \eqref{11} and \eqref{15}.
\end{obs}

\begin{lemm}
\label{l5} Let $u\in\acloc$ satisfy \eqref{13}. Then for every
$\mu\in[0,1)$ we have
\begin{equation}
\label{23} \lim_{t\to
0_+}t^{1-\mu}\int_t^{1/2}\frac{u^2(s)}{s^2}ds=0.
\end{equation}
\end{lemm}

\begin{proof}
Put
$$
f(t)=\frac{u^2(t)}{t}\for t\in(0,1/2],\qquad
g(t)=\frac{1}{t^{\mu}}\for t\in(0,1/2].
$$
Then, obviously, $g\in\lopr$ and in view of \eqref{13},
$f\in\lnek$. Consequently, $fg\in\lopr$. Therefore, for every
$n\in\NN$ there exists $t_n\in (0,1/2]$ such that
\begin{equation}
\label{24} \int_0^{t_n}\frac{u^2(s)}{s^{1+\mu}}ds\leq \frac{1}{n}.
\end{equation}
On the other hand, for any $n\in\NN$, we have
\begin{multline}
0\leq t^{1-\mu}\int_t^{1/2}\frac{u^2(s)}{s^2}ds\leq
t^{1-\mu}\int_{t_n}^{1/2}\frac{u^2(s)}{s^2}ds+t^{1-\mu}\int_t^{t_n}\frac{u^2(s)}{s^{2}}ds\leq\\
\leq
t^{1-\mu}\int_{t_n}^{1/2}\frac{u^2(s)}{s^2}ds+\int_t^{t_n}\frac{s^{1-\mu}u^2(s)}{s^{2}}ds\leq\\
\leq
t^{1-\mu}\int_{t_n}^{1/2}\frac{u^2(s)}{s^2}ds+\int_t^{t_n}\frac{u^2(s)}{s^{1+\mu}}ds
\for t\in (0,t_n],
\end{multline}
and so, in view of \eqref{24},
$$
0\leq \limsup_{t\to
0_+}t^{1-\mu}\int_t^{1/2}\frac{u^2(s)}{s^2}ds\leq
\int_0^{t_n}\frac{u^2(s)}{s^{1+\mu}}ds\leq\frac{1}{n}\for n\in\NN.
$$
Consequently, \eqref{23} holds.
\end{proof}

\begin{lemm}
\label{l6} Let $u\in\acloc$ satisfy \eqref{6} and \eqref{13}. Then
\begin{multline}
\label{25} u(t)=-\left[(1/2-t)\int_0^t
\frac{u^2(s)}{4s}ds+t\int_t^{1/2}
  \frac{u^2(s)}{4s^2}(1/2-s)ds\right.\\
  \left.+\frac{\lambda}{4}t(1/2-t)-2tu(1/2)\right] \mfor t\in(0,1/2],
\end{multline}
\begin{equation}
\label{26} \lim_{t\to 0_+}\frac{u(t)}{t^{\mu}}=0\for \mu\in [0,1),
\end{equation}
and there exists the finite limit
\begin{equation}
\label{27} \lim_{t\to 0_+}\frac{u(t)}{t}<+\infty.
\end{equation}
\end{lemm}

\begin{proof}
For any $\tau\in (0,1/2)$ we have a representation
\begin{multline}
\label{28} u(t)=-\frac{1}{1/2-\tau}\left[(1/2-t)\int_{\tau}^t
  \frac{u^2(s)}{8s^2}(s-\tau)ds+(t-\tau)\int_t^{1/2}
  \frac{u^2(s)}{8s^2}(1/2-s)ds\right.\\
  \left.+\frac{\lambda}{4}(t-\tau)(1/2-t)(1/2-\tau)-u(\tau)(1/2-t)-u(1/2)(t-\tau)\right]
  \mfor t\in [\tau,1/2].
\end{multline}
According to \eqref{13} and Lemma~\ref{l5} we have
\begin{equation}
\label{29} \lim_{\tau\to
  0_+}\int_{\tau}^t\frac{u^2(s)}{s^2}(s-\tau)ds=\lim_{\tau\to
  0_+}\int_{\tau}^t\frac{u^2(s)}{s}ds-\lim_{\tau\to
  0_+}\tau\int_{\tau}^t\frac{u^2(s)}{s^2}ds=\int_0^t\frac{u^2(s)}{s}ds
\end{equation}
and, in view of Remark~\ref{r1}, \eqref{11} holds. Thus, on
account of \eqref{29} and \eqref{11}, from \eqref{28} it follows
that \eqref{25} holds.

Now if we multiply both sides of \eqref{25} by $t^{-\mu}$ and
apply a limit as $t$ tends to zero, with respect to Lemma~\ref{l5}
and \eqref{13}, we obtain that \eqref{26} holds true.

Finally, put
$$
f(t)=\frac{u^2(t)}{t^{1+\mu}}\for t\in(0,1/2],\qquad
g(t)=\frac{1}{t^{1-\mu}}\for t\in(0,1/2].
$$
Then, in view of \eqref{26}, we have $f\in\lnek$ and $g\in\lopr$
provided $\mu\in (0,1)$. Consequently, $fg\in\lopr$, i.e.,
\begin{equation}
\label{30} \int_0^{1/2}\frac{u^2(s)}{s^2}ds<+\infty.
\end{equation}
Now from \eqref{25}, in view of \eqref{13} and \eqref{30}, we get
$$
\lim_{t\to
  0_+}\frac{u(t)}{t}=\left|\int_0^{1/2}\frac{u^2(s)}{4s^2}(1/2-s)ds+\frac{\lambda}{8}-2u(1/2)\right|,
$$
i.e., \eqref{27} holds.
\end{proof}

\section{Upper and Lower Functions}
\label{sec_4}

First we will recall the notion of lower and upper functions to
the general equation
\begin{equation}
\label{31} u''=h(t,u,u'),
\end{equation}
where $h\in\car$ is a Carath\'eodory function.

\begin{Defi}
\label{def1} A continuous function $\gamma: [a,b]\to \RR$ is said
to be a lower (upper) function to \eqref{31} if
$\gamma\in\aclocab$, where $a<t_1<\dots<t_n<b$, there exist finite
limits $\gamma'(t_i+)$, $\gamma'(t_i-)$ $(i=1,\dots,n)$,
$$
\gamma'(t_i-)<\gamma'(t_i+)\qquad
\bigg(\gamma'(t_i-)>\gamma'(t_i+)\bigg)\qquad (i=1,\dots,n),
$$
and the inequality
$$
\gamma''(t)\geq h(t,\gamma(t),\gamma'(t))\qquad
\bigg(\gamma''(t)\leq h(t,\gamma(t),\gamma'(t))\bigg)\forae\tin
$$
holds.
\end{Defi}

The following two lemmas deal with the existence of a solution to
\eqref{31} satisfying the boundary conditions
\begin{equation}
\label{32} u(a)=c_1, \qquad u(b)=c_2,
\end{equation}
and
\begin{equation}
\label{33} u(a)=c_1, \qquad u(b)=u'(b).
\end{equation}
The first one is a simple modification of Scorza Dragoni theorem
and its proof can be found in \cite{vano-boris} (see also the more
recent monograph \cite{DCH}).

\begin{lemm}
\label{l7} Let $\alpha$ and $\beta$ be, respectively, lower and
upper functions to \eqref{31} such that
\begin{equation}
\label{34} \alpha(t)\leq\beta(t)\for\tin
\end{equation}
and
\begin{equation}
\label{35} |h(t,x,y)|\leq q(t)\forae\tin,\quad \alpha(t)\leq x\leq
\beta(t),\quad y\in\RR,
\end{equation}
where $q\in\labrp$. Then, for every $c_1\in[\alpha(a),\beta(a)]$
and $c_2\in[\alpha(b),\beta(b)]$, the problem \eqref{31},
\eqref{32} has a solution $u\in\acabr$ satisfying the condition
\begin{equation}
\label{36} \alpha(t)\leq u(t)\leq \beta(t)\for \tin.
\end{equation}
\end{lemm}

\begin{lemm}
\label{l8} Let $\alpha$ and $\beta$ be, respectively, lower and
upper functions to \eqref{31} satisfying \eqref{34} and
\begin{equation}
\label{37} \alpha(b)\geq \alpha'(b),\qquad \beta(b)\leq \beta'(b).
\end{equation}
Let, moreover, \eqref{35} is fulfilled where $q\in\labrp$. Then,
for every $c_1\in[\alpha(a),\beta(a)]$, the problem \eqref{31},
\eqref{33} has a solution $u\in\acabr$ satisfying \eqref{36}.
\end{lemm}

\begin{proof}
Let $c_1\in[\alpha(a),\beta(a)]$ be arbitrary but fixed. According
to Lemma~\ref{l7} there exists a solution $u_1$ to the equation
\eqref{31} satisfying
\begin{gather}
\label{38}
u_1(a)=c_1,\qquad u_1(b)=\alpha(b),\\
\label{39} \alpha(t)\leq u_1(t)\leq\beta(t)\for\tin.
\end{gather}
On account of \eqref{37}--\eqref{39} we have
$$
u_1(b)\geq u'_1(b).
$$
Furthermore, $u_1$ can be considered as a lower function to
\eqref{31} and so, according to Lemma~\ref{l7}, there exists a
solution $v_1$ to the equation \eqref{31} satisfying
\begin{gather}
\label{41}
v_1(a)=c_1,\qquad v_1(b)=\beta(b),\\
\label{42} u_1(t)\leq v_1(t)\leq\beta(t)\for\tin.
\end{gather}
On account of \eqref{37}, \eqref{41}, and \eqref{42} we have
$$
v_1(b)\leq v'_1(b).
$$
Now we will construct a sequences of solutions to \eqref{31}
$\{u_n\}_{n=1}^{+\infty}$ and $\{v_n\}_{n=1}^{+\infty}$ in the
following way: Having defined solutions $u_n$ and $v_n$ for some
$n\in\NN$ with
\begin{gather}
\label{44} u_n(a)=c_1,\qquad v_n(a)=c_1, \qquad u_n(b)\geq
u'_n(b),\qquad
v_n(b)\leq v'_n(b),\\
\label{45} \alpha(t)\leq u_n(t)\leq v_n(t)\leq\beta(t)\for\tin,
\end{gather}
we can consider them as a lower and an upper function,
respectively, to \eqref{31}. According to Lemma~\ref{l7}, there
exists a solution $u$ to \eqref{31} satisfying
\begin{gather}
\label{46}
u(a)=c_1,\qquad u(b)=\frac{u_n(b)+v_n(b)}{2},\\
\label{47} u_n(t)\leq u(t)\leq v_n(t)\for\tin.
\end{gather}
Obviously, either
\begin{equation}
\label{48} u(b)\leq u'(b)
\end{equation}
or
\begin{equation}
\label{49} u(b)>u'(b).
\end{equation}
If \eqref{48} holds we put
\begin{equation}
\label{50} u_{n+1}(t)=u_n(t)\for\tin, \qquad
v_{n+1}(t)=u(t)\for\tin.
\end{equation}
If \eqref{49} holds we put
\begin{equation}
\label{51} u_{n+1}(t)=u(t)\for\tin, \qquad
v_{n+1}(t)=v_n(t)\for\tin.
\end{equation}
Consequently, in view of \eqref{44}--\eqref{51}, $u_{n+1}$ and
$v_{n+1}$ are solutions to \eqref{31} satisfying
\begin{gather}
u_{n+1}(a)=c_1,\qquad v_{n+1}(a)=c_1, \qquad u_{n+1}(b)\geq
u'_{n+1}(b),\qquad
v_{n+1}(b)\leq v'_{n+1}(b),\nonumber\\
\label{52} \alpha(t)\leq u_n(t)\leq u_{n+1}(t)\leq v_{n+1}(t)\leq
v_n(t)\leq\beta(t)\for\tin.
\end{gather}
Obviously, in view of \eqref{46} and \eqref{50}, resp. \eqref{51},
and \eqref{52}
\begin{equation}
\label{53} \lim_{n\to +\infty}u_n(b)=\lim_{n\to +\infty}v_n(b).
\end{equation}
Moreover, in view of \eqref{44}, for any $n\in\NN$, there exist
$t_n\in[a,b]$ and $s_n\in[a,b]$ such that
$$
u'_n(t_n)=\frac{u_n(b)-c_1}{b-a},\qquad
v'_n(s_n)=\frac{v_n(b)-c_1}{b-a}.
$$
Consequently, on account of \eqref{52}, we have
\begin{equation}
\label{54} |u'_n(t_n)|\leq M,\qquad |v'_n(s_n)|\leq M \for
n\in\NN,
\end{equation}
where
$$
M=\frac{|\alpha(b)|+|\beta(b)|+|c_1|}{b-a}.
$$
Since $u_n$ and $v_n$ are solutions to \eqref{31}, with respect to
\eqref{35}, \eqref{52}, and \eqref{54} we obtain
\begin{gather}
\label{55}
|u'_n(t)|\leq M+\int_a^b q(s)ds\for\tin, \quad n\in\NN,\\
\label{56} |v'_n(t)|\leq M+\int_a^b q(s)ds\for\tin, \quad n\in\NN.
\end{gather}
Thus, on account of \eqref{31}, \eqref{52}, \eqref{55}, and
\eqref{56}, it follows that the sequences
$\{u_n\}_{n=1}^{+\infty}$, $\{u'_n\}_{n=1}^{+\infty}$ and
$\{v_n\}_{n=1}^{+\infty}$, $\{v'_n\}_{n=1}^{+\infty}$ are
uniformly bounded and equicontinuous. Therefore, without loss of
generality we can assume that there exists functions $u_0,v_0\in
C^2\big([a,b];\RR\big)$ such that
$$
u_0^{(j)}(t)=\lim_{n\to +\infty} u_n^{(j)}(t),\qquad
v_0^{(j)}(t)=\lim_{n\to +\infty} v_n^{(j)}(t)\qquad
\mbox{uniformly~on~}\, [a,b]\quad (j=0,1).
$$
By a standard way it can be shown that $u_0,v_0\in\acabr$, and
they are also solutions to \eqref{31}. Moreover, from \eqref{44},
\eqref{52}, and \eqref{53} we have
\begin{gather}
\label{57}
\alpha(t)\leq u_0(t)\leq v_0(t)\leq \beta(t)\for\tin,\\
\label{58} u_0(a)=c_1,\qquad v_0(a)=c_1,\qquad u_0(b)\geq
u'_0(b),\qquad
v_0(b)\leq v'_0(b),\\
\label{59} u_0(b)=v_0(b).
\end{gather}
On the other hand, from \eqref{57} and \eqref{59} it follows that
\begin{equation}
\label{61} u'_0(b)\geq v'_0(b).
\end{equation}
Now \eqref{57}--\eqref{61} imply that $u\stackrel{def}{\equiv}
u_0$ is a solution to \eqref{31}, \eqref{33} satisfying
\eqref{36}.
\end{proof}

\begin{lemm}
\label{l9} Let there exists a function $\alpha\in\acloc$ such that
its restriction to any closed interval $[a,b]\subset (0,1/2]$ is a
lower function to \eqref{6} on $[a,b]$. Let, moreover,
\begin{gather}
\label{62}
\alpha(t)\leq 0\for t\in (0,1/2],\\
\label{63} \lim_{t\to 0_+}\frac{|\alpha(t)|}{t}<+\infty.
\end{gather}
Then there exists a solution $u$ to \eqref{6}, \eqref{8},
\eqref{13} with
\begin{equation}
\label{64} \alpha(t)\leq u(t)\leq 0\for t\in (0,1/2].
\end{equation}
\end{lemm}

\begin{proof}
Note that from \eqref{63} it follows that
\begin{gather}
\label{65} \lim_{t\to 0_+}\alpha(t)=0, \qquad \lim_{t\to
  0_+}\frac{\alpha(t)}{\sqrt{t}}=0,\\
\label{66} \int_0^{1/2}\frac{\alpha^2(s)}{s^2}ds<+\infty.
\end{gather}
Further, from \eqref{65} it follows that
\begin{equation}
\label{67} \alpha^*=\sup\big\{|\alpha(t)|:t\in
(0,1/2]\big\}<+\infty.
\end{equation}
Let $t_n\in (0,1/2)$ for $n\in\NN$ be such that
\begin{equation}
\label{68} t_{n+1}<t_n\for n\in\NN,\qquad \lim_{n\to
+\infty}t_n=0.
\end{equation}
Obviously, for every $n\in\NN$, $\beta\equiv 0$ is an upper
function to \eqref{6} on the interval $[t_n,1/2]$ satisfying
$$
\beta(t_n)=0,\qquad \beta(1/2)=0.
$$
Therefore, according to Lemma~\ref{l7}, in view of \eqref{62}, for
every $n\in\NN$ there exists a solution $u_n$ to \eqref{6} on the
interval $[t_n,1/2]$ satisfying
\begin{gather}
\label{70}
u_n(t_n)=0,\qquad u_n(1/2)=0,\\
\label{71} \alpha(t)\leq u_n(t)\leq 0\for t\in [t_n,1/2].
\end{gather}
Moreover, for every $n\in\NN$ there exists $s_n\in(t_n,1/2)$ such
that
\begin{equation}
\label{72} u'_n(s_n)=0.
\end{equation}
Therefore, integrating \eqref{6} from $s_n$ to $t$, on account of
\eqref{66}, \eqref{71}, and \eqref{72}, we obtain
\begin{multline}
\label{73} |u'_n(t)|=\left|\int_{s_n}^t
  \frac{u^2_n(s)}{8s^2}ds+\frac{\lambda}{2}(t-s_n)\right|\leq
\int_0^{1/2}\frac{\alpha^2(s)}{8s^2}ds+\frac{\lambda}{4}\for
t\in[t_n,1/2],\\ n\in\NN.
\end{multline}
Moreover, from \eqref{6} and \eqref{71} we get
\begin{equation}
\label{74} |u''_n(t)|\leq
\frac{\alpha^2(t)}{8t^2}+\frac{\lambda}{2}\for t\in[t_n,1/2],\quad
n\in\NN.
\end{equation}
Thus, on account of \eqref{66}--\eqref{68}, \eqref{71},
\eqref{73}, and \eqref{74}, we have that the sequences
$\{u_n\}_{n=1}^{+\infty}$, $\{u'_n\}_{n=1}^{+\infty}$ are
uniformly bounded and equicontinuous on every compact subinterval
of $(0,1/2]$. Therefore, according to Arzel\`a--Ascoli theorem,
there exists $u_0\in\cloc$ such that
$$
\lim_{n\to +\infty}u^{(j)}_n(t)=u^{(j)}_0(t)\quad
\mbox{uniformly~on~every~compact~interval~of~}\, (0,1/2]\quad
(j=0,1).
$$
Moreover, since $u_n$ are solutions to \eqref{6}, $u_0\in\acloc$
and it is also a solution to \eqref{6}. Furthermore, from
\eqref{65}, \eqref{70}, and \eqref{71}, it follows that
$$
\alpha(t)\leq u_0(t)\leq 0\for t\in (0,1/2],\qquad
u_0(1/2)=0,\qquad \lim_{t\to 0_+}\frac{u_0(t)}{\sqrt{t}}=0.
$$
\end{proof}

\begin{lemm}
\label{l10} Let the assumptions of Lemma~\ref{l9} be fulfilled.
Let, moreover,
\begin{equation}
\label{77} \alpha(1/2)\geq \alpha'(1/2).
\end{equation}
Then there exists a solution $u$ to \eqref{6}, \eqref{9},
\eqref{13} satisfying \eqref{64}.
\end{lemm}

The proof of Lemma~\ref{l10} is similar to that of Lemma~\ref{l9},
just Lemma~\ref{l8} is used instead of Lemma~\ref{l7} and the
estimate of $u'_n$ is produced using the fact that there exists
$s_n\in (t_n,1/2)$ such that
$$
|u'_n(s_n)|=\frac{|u_n(1/2)|}{1/2-t_n}\leq
\frac{|\alpha(1/2)|}{1/2-t_1}\,.
$$

\section{Lemmas on Estimation of $\lambda$}
\label{sec_5}

\begin{lemm}
\label{l11} The set of numbers $\lambda\geq 0$, for which there
exists a solution to \eqref{6} satisfying \eqref{13} and
\eqref{15}, is nonempty and bounded from above.
\end{lemm}

\begin{proof}
Obviously, for $\lambda=0$ there is a zero solution with the
appropriate properties. Therefore the set is nonempty. If
$\lambda=0$ is the only element of the set, then, clearly, the set
is bounded from above. Let, therefore, $\lambda>0$ and let $u$ be
a solution to \eqref{6} satisfying \eqref{13} and \eqref{15}.
Then, according to Remark~\ref{r1} and Lemma~\ref{l3.5}, we have
that \eqref{7} and \eqref{11} hold.

On the other hand, from \eqref{6} it follows that
\begin{equation}
\label{75} (tu'(t)-u(t))'=\frac{u^2(t)}{8t}+\frac{\lambda}{2}t\for
t\in(0,1/2].
\end{equation}
Integration \eqref{75} from $0$ to $t$, in view of \eqref{7},
\eqref{11}, and \eqref{13}, yields
\begin{equation}
\label{76}
tu'(t)-u(t)=\int_0^t\frac{u^2(s)}{8s}ds+\frac{\lambda}{4}t^2\for
t\in (0,1/2].
\end{equation}
Put
$$
v(t)=-\frac{u(t)}{t}\for t\in (0,1/2].
$$
Then, on account of \eqref{15} and \eqref{76} we have
\begin{gather}
\label{78} v'(t)=-\frac{1}{8t^2}\int_0^t
v^2(s)sds-\frac{\lambda}{4}\for t\in
(0,1/2],\\
\label{79} v(t)\geq 0\for t\in (0,1/2].
\end{gather}
Moreover, from \eqref{78} it follows that $v$ is a decreasing
function and thus \eqref{78} and \eqref{79} result in
$$
v'(t)\leq -\frac{v^2(t)+4\lambda}{16}\for t\in (0,1/2]
$$
whence we get
\begin{equation}
\label{80} -\frac{v'(t)}{v^2(t)+4\lambda}\geq \frac{1}{16}\for
t\in (0,1/2].
\end{equation}
Now the integration of \eqref{80} from $t$ to $1/2$ results in
$$
\int_{v(1/2)}^{v(t)}\frac{dx}{x^2+4\lambda}=-\int_t^{1/2}\frac{v'(s)ds}{v^2(s)+4\lambda}\geq
\frac{1/2-t}{16}\,.
$$
Hence we get
\begin{equation}
\label{81}
\frac{\pi}{4\sqrt{\lambda}}=\int_0^{+\infty}\frac{dx}{x^2+4\lambda}\geq
\lim_{t\to 0_+}\int_{v(1/2)}^{v(t)}\frac{dx}{x^2+4\lambda}\geq
\frac{1}{32}.
\end{equation}
Consequently, \eqref{81} implies $\lambda\leq 64\pi^2$.
\end{proof}

\begin{lemm}
\label{l12} If the problem \eqref{6}, \eqref{8}, \eqref{13} (resp.
\eqref{6}, \eqref{9}, \eqref{13}) is solvable for some
$\lambda_0\geq 0$ then it is solvable also for every $\lambda\in
[0,\lambda_0]$.
\end{lemm}

\begin{proof}
Let $u$ be a solution to the problem \eqref{6}, \eqref{8},
\eqref{13} (resp. \eqref{6}, \eqref{9}, \eqref{13}) with
$\lambda=\lambda_0$. Put $\alpha(t)=u(t)$. Then according to
Remark~\ref{r1}, Lemma~\ref{l2} (resp. Lemma~\ref{l3}), and
Lemma~\ref{l6}, the function $\alpha$ satisfies the assumptions of
Lemma~\ref{l9} (resp. Lemma~\ref{l10}), where the equation
\eqref{6} is considered with $\lambda\in[0,\lambda_0]$.
Consequently, according to Lemma~\ref{l9} (resp. Lemma~\ref{l10}),
the problem \eqref{6}, \eqref{8}, \eqref{13} (resp. \eqref{6},
\eqref{9}, \eqref{13}) is also solvable.
\end{proof}

\begin{lemm}
\label{l13} Let
\begin{equation}
\label{82} \lambda\leq 144.
\end{equation}
Then there exists $\alpha\in\acloc$ satisfying the assumptions of
Lemma~\ref{l9}.
\end{lemm}

\begin{proof}
Put
\begin{equation}
\label{83} \alpha(t)=-48t\left(1-\sqrt{2t}\right)\for t\in
(0,1/2].
\end{equation}
Obviously, \eqref{62} and \eqref{63} hold. We will show that
\begin{equation}
\label{84} \alpha''(t)\geq
\frac{\alpha^2(t)}{8t^2}+\frac{\lambda}{2}\for t\in (0,1/2].
\end{equation}
In view of \eqref{83} we have
$$
\alpha''(t)-\frac{\alpha^2(t)}{8t^2}-72=\frac{72}{\sqrt{2t}}
\left(1-\sqrt{2t}\right)\left(1-2\sqrt{2t}\right)^2\geq 0\for t\in
(0,1/2]
$$
and thus, on account of \eqref{82}, the inequality \eqref{84} is
fulfilled.
\end{proof}

\begin{lemm}
\label{l14} Let
\begin{equation}
\label{85} \lambda\leq 9.
\end{equation}
Then there exists $\alpha\in\acloc$ satisfying the assumptions of
Lemma~\ref{l10}.
\end{lemm}

\begin{proof}
Put
\begin{equation}
\label{86} \alpha(t)=-6t\left(2-\sqrt{2t}\right)\for t\in (0,1/2].
\end{equation}
Obviously, \eqref{62}, \eqref{63}, and \eqref{77} hold. We will
show that \eqref{84} is valid. In view of \eqref{86} we have
$$
\alpha''(t)-\frac{\alpha^2(t)}{8t^2}-\frac{9}{2}=\frac{9}{2\sqrt{2t}}
\left(2-\sqrt{2t}\right)\left(1-\sqrt{2t}\right)^2\geq 0\for t\in
(0,1/2]
$$
and thus, on account of \eqref{85}, the inequality \eqref{84} is
fulfilled.
\end{proof}

\begin{lemm}
\label{l15} Let there exist a function $v\in\acloc\cap\lnek$
satisfying
\begin{gather}
\label{90}
v'(t)=-\frac{1}{8t^2}\int_0^t v^2(s)sds-\frac{\lambda}{4}\for t\in (0,1/2],\\
\label{91}
v(t)\geq 0\for t\in (0,1/2],\\
\label{92} v(1/2)=0.
\end{gather}
Then
\begin{equation}
\label{93} \lambda\leq 384
\end{equation}
and
\begin{equation}
\label{94} v(t)\geq
192\left(1-\sqrt{1-\frac{\lambda}{384}}\right)(1/2-t)\for t\in
(0,1/2].
\end{equation}
\end{lemm}

\begin{proof}
From \eqref{90} it follows that
\begin{equation}
\label{95} v'(t)\leq 0\for t\in (0,1/2].
\end{equation}
Furthermore, $v'\in\acloc$ and
\begin{equation}
\label{96} v''(t)=\frac{1}{4t^3}\int_0^t
v^2(s)sds-\frac{v^2(t)}{8t}\for t\in (0,1/2].
\end{equation}
From \eqref{96}, on account of \eqref{91} and \eqref{95} we obtain
\begin{equation}
\label{97} v''(t)\geq 0\for t\in (0,1/2].
\end{equation}
Therefore, \eqref{90} and \eqref{97} yield
\begin{equation}
\label{98} v'(t)\leq v'(1/2)=-\frac{1}{2}\int_0^{1/2}
v^2(s)sds-\frac{\lambda}{4}\leq -\frac{\lambda}{4}\for t\in
(0,1/2].
\end{equation}
Now the integration of \eqref{98} from $t$ to $1/2$, with respect
to \eqref{92}, results in
\begin{equation}
\label{100} v(t)\geq \frac{\lambda}{4}(1/2-t)\for t\in (0,1/2].
\end{equation}
Put
\begin{equation}
\label{101} c_1=\frac{\lambda}{4},\qquad
c_{n+1}=\frac{c^2_n}{384}+\frac{\lambda}{4}\for n\in\NN.
\end{equation}
We will show that for every $n\in\NN$ the inequality
\begin{equation}
\label{102} v(t)\geq c_n(1/2-t)\for t\in (0,1/2]
\end{equation}
holds. Obviously, \eqref{100} and \eqref{101} yield the validity
of \eqref{102} for $n=1$. Assume therefore that \eqref{102} holds
for some $n\in\NN$. Then from \eqref{98} we obtain
\begin{multline}
\label{103} v'(t)\leq -\frac{1}{2}\int_0^{1/2}
v^2(s)sds-\frac{\lambda}{4}\leq-\left(\frac{c^2_n}{2}\int_0^{1/2}
  (1/2-s)^2sds+\frac{\lambda}{4}\right)=-c_{n+1}\\ \for t\in (0,1/2].
\end{multline}
The integration of \eqref{103} from $t$ to $1/2$, with respect to
\eqref{92}, results in
$$
v(t)\geq c_{n+1}(1/2-t)\for t\in (0,1/2].
$$
Thus \eqref{102} holds for every $n\in\NN$. Moreover,
$$
c_2=\frac{c^2_1}{384}+\frac{\lambda}{4}\geq
\frac{\lambda}{4}=c_1\geq 0,
$$
and, assuming $c_n\geq c_{n-1}\geq 0$ for some $n\in\NN$, we get
$$
c_{n+1}=\frac{c^2_n}{384}+\frac{\lambda}{4}\geq
\frac{c^2_{n-1}}{384}+\frac{\lambda}{4}=c_n.
$$
Thus $\{c_n\}_{n=1}^{+\infty}$ is a non-decreasing sequence of
numbers which are, on account of \eqref{102}, bounded from above.
Therefore, there exists $c_0\in\RR$ such that
$$
c_0=\lim_{n\to +\infty}c_n,
$$
and from \eqref{101} and \eqref{102} we obtain
\begin{gather}
\label{104}
c_0=\frac{c^2_0}{384}+\frac{\lambda}{4},\\
\label{105} v(t)\geq c_0(1/2-t)\for t\in (0,1/2].
\end{gather}
Now \eqref{104} implies \eqref{93} and
$$
c_0\geq 192\left(1-\sqrt{1-\frac{\lambda}{384}}\right)
$$
and, consequently, from \eqref{105} we get \eqref{94}.
\end{proof}

\begin{lemm}
\label{l16} Let there exist a function $v\in\acloc\cap\lnek$
satisfying \eqref{90}, \eqref{91}, and
\begin{equation}
\label{106} v(1/2)=-v'(1/2).
\end{equation}
Then
\begin{equation}
\label{107} \lambda\leq \frac{128}{11}\,.
\end{equation}
\end{lemm}

\begin{proof}
From \eqref{90} it follows that \eqref{95} is fulfilled. Further,
$v'\in\acloc$ and \eqref{96} holds. From \eqref{96}, on account of
\eqref{91} and \eqref{95} we obtain \eqref{97}. Therefore,
\eqref{90} and \eqref{97} yield
\begin{equation}
\label{108} v'(t)\leq v'(1/2)=-c
\end{equation}
where
\begin{equation}
\label{109} c=\frac{1}{2}\int_0^{1/2} v^2(s)sds+\frac{\lambda}{4}.
\end{equation}
Now the integration of \eqref{108} from $t$ to $1/2$, results in
$$
v(1/2)-v(t)\leq -c(1/2-t)\for t\in (0,1/2],
$$
whence, with respect to \eqref{106} and \eqref{108}, we get
\begin{equation}
\label{110} v(t)\geq c(3/2-t)\for t\in (0,1/2].
\end{equation}
Now using \eqref{110} in \eqref{109} we obtain
$$
c\geq \frac{c^2}{2}\int_0^{1/2} (3/2-s)^2 s ds+\frac{\lambda}{4},
$$
i.e.,
\begin{equation}
\label{111} \frac{11}{128}c^2-c+\frac{\lambda}{4}\leq 0.
\end{equation}
However, \eqref{111} implies \eqref{107}.
\end{proof}

\begin{lemm}
\label{l17} Let $u\in\acloc$ satisfy \eqref{6}, \eqref{8}, and
\eqref{13}. Then
$$
\lambda<307.
$$
\end{lemm}

\begin{proof}
Assume that there exists $u\in\acloc$ satisfying \eqref{6},
\eqref{8}, \eqref{13}. According to Remark~\ref{r1} and
Lemma~\ref{l2}, $u$ satisfies \eqref{15}. Moreover, according to
Lemma~\ref{l6}, $u$ satisfies \eqref{25} and \eqref{27}. Put
\begin{equation}
\label{112} v(t)=-\frac{u(t)}{t}\for t\in (0,1/2].
\end{equation}
Then, in view of \eqref{8}, \eqref{15}, \eqref{25}, and \eqref{27}
we have that $v$ satisfies all the assumptions of Lemma~\ref{l15}.
Consequently, \eqref{93} and \eqref{94} hold.

On the other hand, \eqref{8}, \eqref{25}, and \eqref{112} yield
\begin{multline}
\label{113} v(t)=\left[\frac{(1/2-t)}{4t}\int_0^t
  v^2(s)sds+\int_t^{1/2}\frac{v^2(s)}{4}(1/2-s)ds+\frac{\lambda}{4}(1/2-t)\right]\\
\for t\in(0,1/2].
\end{multline}
Obviously, \eqref{90} results in \eqref{95}, and thus, in view of
\eqref{91}, we have
\begin{equation}
\label{114} \int_0^t v^2(s)sds\geq v^2(t)\frac{t^2}{2}\for t\in
(0,1/2].
\end{equation}
Now using \eqref{94} and \eqref{114} in \eqref{113} we obtain
\begin{equation}
\label{115} v(t)\geq
\frac{(1/2-t)t}{8}v^2(t)+\frac{c^2}{4}\int_t^{1/2}(1/2-s)^3ds+\frac{\lambda}{4}(1/2-t)\mfor
t\in (0,1/2],
\end{equation}
where
\begin{equation}
\label{116} c=192\left(1-\sqrt{1-\frac{\lambda}{384}}\right).
\end{equation}
The inequality \eqref{115} means that for every fixed $t\in
(0,1/2]$, the value of the function $v$ at the point $t$ is a
solution to the quadratic inequality
\begin{equation}
\label{117}
\frac{(1/2-t)t}{8}x^2-x+\frac{(1/2-t)}{4}\left[\frac{c^2(1/2-t)^3}{4}+\lambda\right]\leq
0.
\end{equation}
Thus \eqref{117} results in
$$
f(\lambda,t)\stackrel{def}{=}\frac{(1/2-t)^2t}{8}\left[\frac{c^2(1/2-t)^3}{4}+\lambda\right]\leq
1\for t\in (0,1/2],
$$
with $c$ given by \eqref{116}. However, it can be verified by a
direct calculation that
$$
f(307,1/8)>1.
$$
Therefore, if $\lambda=307$, there is no solution to \eqref{6},
\eqref{8}, \eqref{13}, and, according to Lemma~\ref{l12}, there is
no solution to \eqref{6}, \eqref{8}, \eqref{13} for $\lambda\geq
307$ either.
\end{proof}

\begin{lemm}
\label{l18} Let $u\in\acloc$ satisfy \eqref{6}, \eqref{9}, and
\eqref{13}. Then \eqref{107} is fulfilled.
\end{lemm}

\begin{proof}
Assume that there exists $u\in\acloc$ satisfying \eqref{6},
\eqref{9}, \eqref{13}. According to Remark~\ref{r1} and
Lemma~\ref{l3}, $u$ satisfies \eqref{15}. Moreover, according to
Lemma~\ref{l6}, $u$ satisfies \eqref{25} and \eqref{27}. Define
$v$ by \eqref{112}. Then, in view of \eqref{9}, \eqref{15},
\eqref{25}, and \eqref{27} we have that $v$ satisfies all the
assumptions of Lemma~\ref{l16}. Consequently, \eqref{107} holds.
\end{proof}

\section{Proofs of the Main Results}
\label{sec_6}

Theorem~\ref{t1} (resp. \ref{t2}) follows from Remark~\ref{r1},
Lemmas~\ref{l2} (resp. \ref{l3}), \ref{l4}, \ref{l11}, \ref{l12},
and transformation \eqref{5}. Proposition~\ref{p1} follows from
Lemmas~\ref{l4}, \ref{l9}, \ref{l13}, \ref{l17}, and
transformation \eqref{5}. Proposition~\ref{p2} follows from
Lemmas~\ref{l4}, \ref{l10}, \ref{l14}, \ref{l18}, and
transformation \eqref{5}.

\end{document}